\documentclass[a4paper,oneside,11pt]{amsart}

\reversemarginpar

\usepackage[colorlinks,hyperindex,linkcolor=blue,urlcolor=black,pdftitle={Non-stable subnormal contractions have nontrivial hyperinvariant subspaces},pdfauthor={Maria F. Gamal'}]
{hyperref}

\usepackage{amsmath}
\usepackage{amsfonts}
\usepackage{amssymb}
\usepackage{amsthm}

\usepackage[english]{babel}
\usepackage{enumerate}

\usepackage{graphicx}
\usepackage{color}

\usepackage[abbrev]{amsrefs}

\numberwithin{equation}{section}

\theoremstyle{plain}
\newtheorem{theorem}{Theorem}[section]
\newtheorem{lemma}[theorem]{Lemma}
\newtheorem{corollary}[theorem]{Corollary}
\newtheorem{proposition}[theorem]{Proposition}

\theoremstyle{definition}
\newtheorem{remark}[theorem]{Remark}

\begin{document}

\title[Existence of hyperinvariant subspaces]{Non-stable subnormal contractions have nontrivial hyperinvariant subspaces}

\author{Maria F. Gamal'}
\address{
MARIA F. GAMAL', 
St. Petersburg Branch\\ V. A. Steklov Institute 
of Mathematics\\
 Russian Academy of Sciences\\ Fontanka 27, St. Petersburg\\ 
191023, Russia  
}
\email{gamal@pdmi.ras.ru}

\begin{abstract}
A contraction $T$ on a (complex, separable) Hilbert space  is stable, or of class $C_{0\cdot}$, 
if $T^n\to 0$ in the strong operator topology. It is proved that for  a \emph{non-stable} pure subnormal contraction $T$ 
there exists a singular inner function $\theta$ such that the range of $\theta(T)$ is not dense. Consequently, $T$ has 
 nontrivial hyperinvariant subspaces. The proof is based on results by Esterle and K\'erchy. 
Examples of stable subnormal contractions are given for which the range of $\varphi(T)$ is dense
 for every $\varphi\in H^\infty$
($\varphi\not\equiv 0$). 
\smallskip

\noindent \textsc{Keywords:} \emph{Hyperinvariant subspace,  subnormal operator, contraction.}

\smallskip

\noindent MSC (2020): 47A15, 47A45,  47A60, 47B20.

 \end{abstract}

\maketitle

\section{Introduction}

Let $\mathcal H$ be a (complex, separable) Hilbert space, and let  $\mathcal L(\mathcal H)$ be the algebra of all (bounded linear) operators acting on  $\mathcal H$.
 The algebra of all $R\in\mathcal L(\mathcal H)$  such that $TR=RT$ is called the \emph{commutant} of $T$ and is denoted by  $\{T\}'$.  A (closed) subspace  $\mathcal M$ of  $\mathcal H$ is called \emph{invariant} 
for an operator  $T\in\mathcal L(\mathcal H)$, if $T\mathcal M\subset\mathcal M$, and \emph{hyperinvariant} 
for $T$ if $R\mathcal M\subset\mathcal M$ for all $R\in\{T\}'$. The complete lattice of all invariant (resp.,  hyperinvariant) subspaces for  $T$ is denoted by  $\operatorname{Lat}T$ (resp., by 
$\operatorname{Hlat}T$). 
 The \emph{hyperinvariant subspace problem}  is the question whether for every nontrivial operator 
$T\in\mathcal L(\mathcal H)$ there exists a nontrivial hyperinvariant subspace. 
Here ``nontrivial operator'' means  not a scalar multiple of the identity operator $I_{\mathcal H}$, 
and ``nontrivial subspace'' means  different from  $\{0\}$ and  $\mathcal H$.

A (closed) subspace $\mathcal M$ is called \emph{reducing} for $T$ if 
 $\mathcal M\in\operatorname{Lat}T\cap\operatorname{Lat}T^*$. 
It is known 
 that if an operator $T$ has a (non-zero) reducing subspace  $\mathcal M$ such that 
 $T|_{\mathcal M}$ has a nontrivial hyperinvariant subspace or $T|_{\mathcal M}=\lambda I_{\mathcal M}$ for some 
$\lambda\in \mathbb C$, then $T$ has  a nontrivial hyperinvariant subspace (if $T\neq\lambda I_{\mathcal H}$), see \cite{douglaspearcy}. 
 In particular, this is fulfilled if 
  $T|_{\mathcal M}$ is a normal operator,  because a normal operator has nontrivial hyperinvariant subspaces 
or is a scalar multiple of the identity operator.

Recall that an operator $T\in\mathcal L(\mathcal H)$ is called \emph{subnormal}
 if there exists a complex Hilbert space $\mathcal K$ and a normal operator $N\in\mathcal L(\mathcal K)$ 
 such that $\mathcal H\subset\mathcal K$, $\mathcal H\in\operatorname{Lat}N$ and $T=N|_{\mathcal H}$.
$N$ is a minimal normal extension of $T$, if $\vee_{n=0}^\infty N^{*n}\mathcal H=\mathcal K$. 
 Every subnormal
operator $T$ has  a unique (up to unitary equivalence) minimal normal
extension, see Corollary II.2.7 in \cite{conwaysubnormal}.  
If $T$ is subnormal,  $\mathcal M\in\operatorname{Lat}T$ and $T|_{\mathcal M}$ is normal, then $\mathcal M$ is a reducing subspace
for $T$. A subnormal operator $T$ is called \emph{pure}
 if it has no (non-zero) reducing subspaces $\mathcal M$ such that $T|_{\mathcal M}$ is normal. 
Every subnormal operator $T$ has the form $T=N\oplus T_0$, where $N$ is normal and $T_0$ is  pure subnormal 
(one of summands can act on the zero space; see Sec. II in  \cite{conwaysubnormal}).

It is known that subnormal operator is reflexive (see Theorem VII.8.5 in \cite{conwaysubnormal} or \cite{olinthomson}), 
 and rationally cyclic subnormal operator  has nontrivial hyperinvariant subspaces
(see Corollary V.4.7 in \cite{conwaysubnormal} or \cite{thomson}). For some other results, see \cite{foiasjungkopearcy}. 
But (up the author's knowledge) 
the existence of nontrivial hyperinvariant subspaces for arbitrary subnormal operator is unknown.

An operator $T\in\mathcal L(\mathcal H)$  is called 
a  \emph{contraction} if $\|T\|\leq 1$.  It is easy to see that for a contraction $T$ the \emph{stable subspace}  
\begin{equation}\label{hhtt0}\mathcal H_{T,0}=\{x\in\mathcal H :\ \|T^nx\|\to 0\}
\end{equation}
 is hyperinvariant for $T$ (Theorem II.5.4 in \cite{sznfbk}). 
The classes $C_{ab}$ of contractions, where 
 $a$ and $b$ can be $0$, $1$, or a dot,  
are defined as follows. If $\mathcal H_{T,0}=\mathcal H$, then  $T$ is \emph{stable}, in other words, $T$ is  \emph{of class} $C_{0\cdot}$, while if  $\mathcal H_{T,0}=\{0\}$, then $T$ is 
 \emph{of class} $C_{1\cdot}$. Furthermore,  $T$  is \emph{of class} $C_{\cdot a}$, if $T^*$ is of class  $C_{a\cdot}$,  
 and $T$ is  \emph{of class} $C_{ab}$, if $T$ is of class $C_{a\cdot}$ and of class $C_{\cdot b}$, $a$, $b=0,1$.

Denote by $\mathbb D$ and  $\mathbb T$ the open unit disc and the unit circle, respectively, 
and let $m$ be the normalized Lebesgue measure on $\mathbb T$. 
Let $T\in\mathcal L(\mathcal H)$ be a subnormal contraction and let $N_T$ be a minimal normal extension of $T$. By Theorem II.2.11 in
 \cite{conwaysubnormal}, $N$ is a contraction, too. Consequently, $N_T=U\oplus N$, where $U$ is a unitary operator and $N$ is a normal operator whose spectral measure is carried on $\mathbb D$; one of summands can act on the zero space. 
Since $N$ is  stable (moreover, $N$ is of class $C_{00}$),  the equality $N_T=N$ implies that $T$ 
 is of class $C_{00}$. Thus, if $T$ is a non-stable subnormal contraction, then  $U$ acts on a non-zero space. 

Every unitary operator $U$ has the form $U=U_{(s)}\oplus U_{(a)}$, where $U_{(a)}$ is \emph{absolutely continuous}, that is,  the spectral measure of $U_{(a)}$ is absolutely continuous with respect to $m$, and $ U_{(s)}$
is \emph{singular}, that is,  the spectral measure of $U_{(s)}$ is singular   with respect to $m$. (One of summands can act on the zero space.) Thus, 
\begin{equation*}  N_T=U_{(s)}\oplus U_{(a)}\oplus N \ \text{ and }\ T=N_T|_{\mathcal H}, 
\ \text{ where }\ \mathcal H\in\operatorname{Lat}N_T. 
\end{equation*}

It is well-known (see, for example, Sec. 3 and 4 in \cite{k16})
that 
$ \mathcal H$ has the form $ \mathcal  H=\mathcal H_{(s)}\oplus\mathcal H_{(a)}$, 
where $\mathcal H_{(s)}\in\operatorname{Lat} U_{(s)}$ and $\mathcal H_{(a)}\in\operatorname{Lat}(U_{(a)}\oplus N)$. 
Since all invariant subspaces of $ U_{(s)}$ are reducing,  $ \mathcal H_{(s)}$ is a reducing subspace for $U_{(s)}$. 
 Since $N_T$ is the minimal normal extension of $T$,   $ \mathcal H_{(s)}$ is the whole space on which  $U_{(s)}$ acts. Furthermore,  $ \mathcal H_{(s)}$ is a reducing subspace 
for $T$.  If  $ \mathcal H_{(s)}\neq\{0\}$, then $T$ has a nontrivial hyperinvariant subspace. 
Thus, assuming that $T$ is a \emph{non-stable} subnormal contraction,  it is sufficient to consider such $T$ whose minimal normal extension has the form $U\oplus N$, where $U$ is an absolutely continuous unitary operator and $N$ is  a normal operator whose spectral measure is carried on $\mathbb D$. The same can be deduced from Theorems I.3.2 and II.6.4, Proposition IX.1.7 in \cite{sznfbk} and Lemma \ref{asymptote} below. 

Let $H^\infty$ be the algebra of all analytic bounded functions in $\mathbb D$, and let $U$ and $N$ be as just above. 
 Let $\varphi\in H^\infty$. Then $\varphi(U\oplus N)$ is well defined, 
 $\mathcal H\in \operatorname{Lat}\varphi( U\oplus N)$, and 
$\varphi(T)=\varphi(U\oplus N)|_{\mathcal H}$. 

In the present paper, 
 the existence of a nontrivial hyperinvariant subspace is proved for \emph{non-stable} subnormal contractions. 
Namely, it is proved that if $T$ is a pure subnormal contraction and $T$ is \emph{not} of class $C_{0\cdot}$, 
then there exists a singular inner function $\theta$ such that the range of $\theta(T)$ is not dense (Theorem \ref{thm12}).   In  Sec. 4 examples of stable subnormal contractions $T$ are given for which the range of $\varphi(T)$ is dense 
for every $\varphi\in H^\infty$ ($\varphi\not\equiv 0$) (Corollary \ref{cor44}). 

In the next section, the definitions are recalled and the facts needed in the sequel are collected. 

\section{Definitions and preliminaries}
Symbols $\mathbb D$, $\mathbb D^-$, and $\mathbb T$ denote the open unit disc, the closed unit disc, 
and the unit circle, respectively. The normalized Lebesgue measure on $\mathbb T$ is denoted by $m$. 
For a Borel set $\tau\subset\mathbb T$ 
denote by $U_\tau$ the operator of multiplication by the independent variable acting on $L^2(\tau,m)$.
 The unilateral shift $S$ is the restriction of $U_{\mathbb T}$ on the Hardy space $H^2$ on 
$\mathbb D$ considered as a subspace of $L^2(\mathbb T,m)$.  Furthermore,  
$H^2_-=L^2(\mathbb T,m)\ominus H^2$, and $H^\infty$ is the algebra of all analytic bounded functions in $\mathbb D$. 

For a Hilbert space  $\mathcal H$ and a (closed) subspace  $\mathcal M$ of $\mathcal H$, the symbol  $P_{\mathcal M}$   denotes the orthogonal projection onto $\mathcal M$. The spectrum, the  point spectrum, and the spectral radius  of an operator $T$ are denoted by $\sigma(T)$,   
$\sigma_{\mathrm{p}}(T)$, and $r(T)$, respectively. 

Let  $\mathcal H$ and $\mathcal K$ be Hilbert spaces,   let $T\in\mathcal L(\mathcal H)$ and  $R\in\mathcal L(\mathcal K)$, 
and let $X$ be a (linear, bounded) transformation acting from $\mathcal H$ to $\mathcal  K$ such that $XT=RX$. 
If $X$ is a unitary transformation, then $T$ and $R$ 
are called  \emph{unitarily equivalent}, in notation: $T\cong R$.  If $X$ is invertible, that is, $X^{-1}$ is bounded, 
then $T$ and $R$ are called \emph{similar}, in notation: $T\approx R$.

\subsection{Isometric and unitary asymptotes}
 
To every contraction $T\in\mathcal L(\mathcal H)$ it can be associated an \emph{isometric asymptote} $(X_+, V)$ of $T$, 
where $V$ is an isometry and $X_+$ is a canonical intertwining mapping: $X_+T=VX_+$. The \emph{unitary asymptote} 
$(X,U)$ of $T$  is the minimal unitary extension of 
the  isometric asymptote of $T$. More precisely, a unitary operator $U$ is  the minimal unitary extension of $V$, 
and $X$ is a natural extension of $X_+$. For exact definition of isometric and unitary asymptote 
 see  \cite{k89}, or Sec. IX.1 in  \cite{sznfbk}, or Sec. 2 in \cite{k16}. 
The properties (i)-(v) of isometric and unitary asymptotes below will be used. The properties (i)-(iv) follow from the construction given in 
 \cite{k89} and in Sec. IX.1 from \cite{sznfbk}. Property (v) is a particular case of Theorem IX.3.5 from \cite{sznfbk}, see also \cite{k07}. 

(i) $\mathcal H_{T,0}=\ker X_+$, where $\mathcal H_{T,0}$ is the stable subspace of $T$ defined in \eqref{hhtt0}.

(ii) The range of $X_+$ is dense.

(iii) If $\mathcal M\in\operatorname{Lat}T$, then  $(X_+|_{\mathcal M}, V|_{\operatorname{clos}X_+\mathcal M})$ 
is an isometric asymptote of $T|_{\mathcal M}$.

(iv) If $\mathcal M\in\operatorname{Lat}T$ is such that $T|_{\mathcal M}$ is similar to an isometry, 
then  $X_+|_{\mathcal M}$ is bounded below. Consequently,  $X_+\mathcal M$ is closed and 
$T|_{\mathcal M}\approx V|_{X_+\mathcal M}$. 

(v) If a scalar-valued spectral measure of $U$ is $m$, then there exists  
$\mathcal M\in\operatorname{Lat}T$  such that $T|_{\mathcal M}\approx S$.  

Property (ii) of isometric asymptote has the following corollary: 
if the isometry $V$ from an isometric asymptote of a contraction $T$ is not unitary, 
then there exists a (linear, bounded) transfomation $Y$ such that $YT=SY$ and 
 $\operatorname{clos}Y\mathcal H=H^2$. This implies $\sigma_{\mathrm{p}}(T^*)=\mathbb D$. 
Moreover, if this  contraction $T$ does not have a singular unitary summand, then 
$\operatorname{clos}\theta(T)\mathcal H\neq\mathcal H$ for every inner function $\theta$, see Lemma 1.3 in \cite{g19}.

\begin{lemma}\label{asymptote} Suppose that  $\mathcal G$ and $\mathcal K$ are Hilbert spaces, 
$U\in \mathcal L(\mathcal G)$ is a unitary operator, and $N\in \mathcal L(\mathcal K)$ is 
a stable (that is, of class $C_{0\cdot}$) contraction. 
Let $\mathcal H\in\operatorname{Lat}(U\oplus N)$. Set $T=(U\oplus N)|_{\mathcal H}$. Then 
\begin{equation*} (P_{\mathcal G}|_{\mathcal H}, U|_{\operatorname{clos}P_{\mathcal G}\mathcal H})
 \end{equation*}
is an isometric asymptote of $T$.

Moreover, if $\mathcal M\in\operatorname{Lat}(U\oplus N)$ 
is such that  $(U\oplus N)|_{\mathcal M}$ is similar to an isometry, then
$P_{\mathcal G}|_{\mathcal M}$ is bounded below.  Consequently, $P_{\mathcal G}{\mathcal M}$ is closed.  
\end{lemma}

\begin{proof} The statement of the lemma follows from the construction of isometric asymptote described in 
\cite{k89} or in Sec. IX.1 from \cite{sznfbk}. Namely, a new semi-inner product $\langle \cdot,\cdot\rangle$ on $\mathcal G\oplus\mathcal K$ 
used in this construction 
 is defined by the formula 
\begin{equation*}\langle x,y\rangle=\lim_{n\to\infty}((U\oplus N)^nx,(U\oplus N)^ny), \ \text{ where }\ 
x, y\in \mathcal G\oplus\mathcal K.\end{equation*} 
Assumptions on $U$ and $N$ imply 
that $\langle x,y\rangle=(P_{\mathcal G}x,P_{\mathcal G}y)$. Therefore, an isometric asymptote of $U\oplus N$ 
is $ (P_{\mathcal G}, U)$. The last relation and properties (iii) and (iv) of isometric asymptote imply
 the conclusion of the lemma.
\end{proof}
   
\subsection{Lemmas from operator theory}

\begin{lemma}\label{lemminimal} Suppose that $\mathcal G$ and $\mathcal K$ are  Hilbert spaces, and  $U\in \mathcal L(\mathcal G)$ and 
$N\in \mathcal L(\mathcal K)$  are normal operators. 
Let $\mathcal H\in\operatorname{Lat}(U\oplus N)$ be such that
\begin{equation} \label{minimal}\vee_{n=0}^\infty (U\oplus N)^{*n}\mathcal H=\mathcal G\oplus\mathcal K .
 \end{equation}
Then  $\vee_{n=0}^\infty U^{*n}\operatorname{clos}P_{\mathcal G}\mathcal H=\mathcal G$.
\end{lemma}

\begin{proof}
Set $\mathcal G_1=\vee_{n=0}^\infty U^{*n}\operatorname{clos}P_{\mathcal G}\mathcal H$.  
Then $\mathcal G_1$ is a reducing subspace for $U$. Consequently, $\mathcal G_1\oplus\mathcal K$ is 
a reducing subspace for $U\oplus N$. 
Since 
\begin{equation*} \mathcal H\subset\operatorname{clos}P_{\mathcal G}\mathcal H\oplus\operatorname{clos}
P_{\mathcal K}\mathcal H
\subset\mathcal G_1\oplus\mathcal K, 
 \end{equation*}
we conclude that 
$\vee_{n=0}^\infty (U\oplus N)^{*n}\mathcal H\subset\mathcal G_1\oplus\mathcal K$.  
This implies $\mathcal G_1=\mathcal G$. 
\end{proof}

\begin{lemma}\label{lemtau}  Suppose that $T$ is a contraction,  $\tau\subset\mathbb T$ is a Borel set,  $m(\tau)<1$, 
and $U_\tau$ is the operator of multiplication by the independent variable acting on $L^2(\tau,m)$.   
Set 
\begin{equation*}  R=U_{\tau}\oplus T.
\end{equation*}
If there exists  $\mathcal M\in\operatorname{Lat}R$ such that  $R|_{\mathcal M}\cong U_{\mathbb T}$, then 
there exists a reducing subspace $\mathcal M_0$ for $T$ such that $\mathcal M_0\neq\{0\}$ and
   $T|_{\mathcal M_0}$ is unitary. 
\end{lemma}
\begin{proof} Denote by $\mathcal H$ the space on which $T$ acts.
Set \begin{equation*} \mathcal M_0=\{x\in\mathcal H\ \colon \ \|T^nx\|=\|x\|=\|T^{*n}x\| 
\text{ for every } n\in\mathbb N\}.
 \end{equation*}
 By Theorem I.3.2 in \cite{sznfbk}, $\mathcal M_0$ is a reducing subspace for $T$, and 
$T|_{\mathcal M_0}$ is unitary. It needs to prove that $\mathcal M_0\neq\{0\}$.  

Since $R$ is a contraction and  $R|_{\mathcal M}$ is unitary, 
 $\mathcal M$ is a reducing subspace for $R$. Consequently,  $R^*|_{\mathcal M}$ is unitary. 
Let $\psi\in L^2(\tau,m)$ and $x\in\mathcal H$ be such that $\psi\oplus x\in\mathcal M$.
Then for every $n\in\mathbb N$ 
\begin{align*} 
\|U_{\tau}^n\psi\|^2+\|T^nx\|^2&=\|R^n(\psi\oplus x)\|^2=\|\psi\oplus x\|^2=\|R^{*n}(\psi\oplus x)\|^2
\\&=
\|U_{\tau}^{-n}\psi\|^2+\|T^{*n}x\|^2.
\end{align*} 
Consequently, $x\in\mathcal M_0$. If $x=0$ for every $\psi\in L^2(\tau,m)$ 
such that $\psi\oplus x\in\mathcal M$, then $\mathcal M\subset  L^2(\tau,m)$, 
a contradiction with the assumption $m(\tau)<1$. Thus,  $\mathcal M_0\neq\{0\}$. 
\end{proof}

\subsection{Lemmas from function theory}

Let $\varphi\colon\mathbb D\to\mathbb C$ be an analytic function. For $\xi\in\mathbb D^-$ set 
$\varphi_\xi(z)=\varphi(\xi z)$, $z\in\mathbb D$. If $\varphi\in H^\infty$, then  $\varphi_\xi\in H^\infty$ for every 
 $\xi\in\mathbb D^-$. Consequently, if $\xi\in\mathbb T$, then  
$\varphi_\xi(\zeta)=\varphi(\xi\zeta)$ is defined for $m$-a.e. $\zeta\in\mathbb T$. 

Recall that a singular inner function $\vartheta$ is a function from $H^\infty$ such that $\vartheta(\zeta)=1$  
for $m$-a.e. $\zeta\in\mathbb T$ and $\vartheta(z)\neq 0$ for every $z\in\mathbb D$. It is assuming that $\vartheta$   is not a constant function. In particular, $|\vartheta(0)|<1$.  

The following lemmas are well-known; the details of the proofs can be found in the version of \cite{g19} on arXiv.org.
 
\begin{lemma}\label{lemmutually} Suppose that $\vartheta$ is a singular inner function and $h\in H^2$. If 
$ h\in\vartheta_\xi H^2$ for every $\xi\in\mathbb T$, then $h\equiv 0$. 
\end{lemma}

\begin{lemma}\label{lemphitt} Let $\varphi\in H^\infty$, and let $f\in L^1(\mathbb T,m)$. Set 
 \begin{equation*} 
(\varphi\ast f)(\xi)=\int_{\mathbb T}\varphi_\xi f\mathrm{d}m, \ \  \xi\in \mathbb D^-. \end{equation*} 
Then $\varphi\ast f$ is analytic in $\mathbb D$, continuous in $\mathbb D^-$, and 
\begin{equation*}\widehat{\varphi\ast f}(n)=\widehat\varphi(n)\widehat f(-n),  \ \ \ n\geq 0.\end{equation*}
\end{lemma}

The following lemma is  needed.

\begin{lemma}\label{lemphidd} 
Let $\mu$ be a positive finite Borel measure on $\mathbb D$,  let $F\in L^1(\mathbb D,\mu)$, 
and let $\varphi\colon\mathbb D\to\mathbb C$ be an analytic function.
Set 
 \begin{equation*} 
\Phi(\xi)=\int_{\mathbb D}\varphi_\xi F\mathrm{d}\mu, \ \  \xi\in \mathbb D . \end{equation*} 
Then $\Phi$ is analytic in $\mathbb D$  
 and
\begin{equation*}\widehat{\Phi}(n)=\widehat\varphi(n)\int_{\mathbb D}z^n F(z)\mathrm{d}\mu(z), \ \ \ n\geq 0.\end{equation*}
\end{lemma}

\begin{proof} Since $\varphi$ is analytic in $\mathbb D$, we have  $\limsup_n|\widehat\varphi(n)|^{1/n}\leq 1$ and 
$\varphi_\xi(z)=\sum_{n=0}^\infty\widehat\varphi(n)\xi^nz^n$, $z\in\mathbb D$. Consequently, 
\begin{equation*} 
\int_{\mathbb D}\sum_{n=0}^\infty|\widehat\varphi(n)||\xi|^n |z|^n | F(z)|\mathrm{d}\mu(z)<\infty, \end{equation*} 
because $\xi\in\mathbb D$, and the statement of the lemma follows.
\end{proof}

The following lemma is well-known;  the proof is contained in the proof of Lemma 2.6 in \cite{gamaljot}.

\begin{lemma}\label{lemalpha} Let $\alpha$ be a positive finite Borel measure on $\mathbb D$. 
Then there exists a  continuous, strongly increasing function  
 $u\colon[0,1)\to (0,\infty)$ such that  $u(r)\to\infty$ when $r\to 1$ and 
\begin{equation*} \int_{\mathbb D}u(|z|)^2\mathrm{d}\alpha(z)<\infty.\end{equation*}
\end{lemma}

The following lemma is Lemma 5.6 in \cite{est} cited for a convenience.

\begin{lemma}[\cite{est}]\label{lemest} Let $u\colon[0,1)\to (0,\infty)$ be a continuous, strongly increasing function such that 
 $u(r)\to\infty$ when $r\to 1$. 
Then there exists a singular inner function $\vartheta$ such that 
\begin{equation*} \sup_{z\in\mathbb D}\frac{1}{u(|z|)|\vartheta(z)|}=C<\infty.\end{equation*}
\end{lemma}

The proof of the following lemma  is contained in the proof of Lemma 2.6 in \cite{gamaljot}.

\begin{lemma}\label{lemsubspace}
Let  $\mu$ be  a positive finite  Borel  measure on $\mathbb D$, 
and 
\begin{equation*}N = \oplus_{n=1}^\infty N_{\mu|\Delta_n},
\end{equation*}
where $\Delta_n\subset\mathbb D$ are Borel sets and $N_{\mu|\Delta_n}$ is the operator 
of multiplication by the independent variable  
on $L^2(\Delta_n,\mu)$. 
(Note that it is not assumed here that $\Delta_n$ are disjoint; moreover, it is possible that 
$\Delta_n=\Delta_k$ for some $n\neq k$. On the other hand, it is possible that $\Delta_n=\emptyset$ for sufficiently large $n$.)

Let $f=\oplus_{n=1}^\infty f_n\in\oplus_{n=1}^\infty L^2(\Delta_n,\mu)$, where $f_n\in L^2(\Delta_n,\mu)$. 
Set 
\begin{equation*}
\mathrm{d}\alpha(z)=\Big(\sum_{n=1}^\infty|f_n(z)|^2\Big)\mathrm{d}\mu(z).
\end{equation*}
Let $u$ be a function from Lemma \ref{lemalpha} applied to $\alpha$, and let  $\vartheta$ be a singular inner function 
from  Lemma \ref{lemest} applied to $u$. 
Set $g= \oplus_{n=1}^\infty f_n/\vartheta$. Let $\{r_n\}_{n=1}^\infty\subset(0,1)$ be such
$r_n\to 1$. Set $\varphi_n(z)=1/\vartheta(r_nz)$,  $z\in\mathbb D$.  Then  $\varphi_n\in H^\infty$ 
for every $n\geq 1$
and $g=\lim_n\varphi_n(N)f$. 
\end{lemma} 

The following lemma is  needed.

\begin{lemma}\label{lemthetadd} 
 Let $u$ and $\vartheta$ be as in Lemma \ref{lemest}. 
Suppose that $\mu$ is a positive finite Borel measure on $\mathbb D$, 
 $F\in L^1(\mathbb D,\mu)$, and  
\begin{equation*}  \int_{\mathbb D}u(|z|)||F(z)|\mathrm{d}\mu(z)<\infty.
 \end{equation*} 
For $\xi\in\mathbb D^-$ set 
\begin{equation*} G(\xi)=\int_{\mathbb D}\frac{F}{\vartheta_\xi}\mathrm{d}\mu. \end{equation*} 
Then $G$ is analytic in $\mathbb D$,  continuous in $\mathbb D^-$
and 
\begin{align}\label{ggdd}\widehat G(n)=\widehat{\frac{1}{\vartheta}}(n)
\int_{\mathbb D}z^n F(z)\mathrm{d}\mu(z), \ \ \ n\geq 0.\end{align}
\end{lemma}

\begin{proof} Since $\frac{1}{\vartheta}$ is analytic in $\mathbb D$, Lemma \ref{lemphidd} is applied. Therefore,
  $G$ is analytic in $\mathbb D$, and \eqref{ggdd} is fulfilled. 

Let $\xi_0\in\mathbb T$, $\{\xi_j\}_{j=1}^\infty\subset\mathbb D^-$, and $\xi_j\to\xi_0$. 
Then  
\begin{equation*}\frac{|F(z)|}{|\vartheta_{\xi_j}(z)|}\leq C|F(z)|u(|\xi_jz|)\leq C|F(z)|u(|z|).
 \end{equation*} 
Since $\vartheta(\xi_j z)\to \vartheta(\xi_0 z)$ for every $z\in\mathbb D$, the Lebesgue convergent  theorem implies 
 $G(\xi_j)\to G(\xi_0)$. Consequently, $G$ is continuous at every $\xi_0\in\mathbb T$. 
Thus, $G$ is continuous in $\mathbb D^-$.
\end{proof}

\section{Main result}

Recall that $m$ is the normalized Lebesgue measure on the unit circle $\mathbb T$.

\begin{theorem}\label{thm11} Suppose that $T$ is a non-stable (that is, not of class $C_{0\cdot}$) subnormal contraction, 
$U$ is a unitary summand of the minimal normal extension of $T$, and $m$ is a scalar-valued spectral measure for $U$. 
Then  there exists $\mathcal M\in \operatorname{Lat}T$ such that  $T|_{\mathcal M}\cong U_{\mathbb T}$ or there exists 
a singular inner function $\theta$ such that the range of $\theta(T)$ is not dense.
\end{theorem}

\begin{proof} Denote by $N$ the stable summand of the minimal normal extension of $T$, denote by $\mathcal K$
 and $\mathcal G$ the spaces on which $N$ and $U$ act, respectively,
and by $\mathcal H$ the space on which $T$ acts. Then  $\mathcal H\in\operatorname{Lat}(U\oplus N)$, 
$T=(U\oplus N)|_{\mathcal H}$ and \eqref{minimal} is fulfilled. 
 
By Lemma \ref{asymptote}, 
\begin{equation*} (P_{\mathcal G}|_{\mathcal H}, U|_{\operatorname{clos}P_{\mathcal G}\mathcal H})
 \end{equation*}
is an isometric asymptote of $T$.
If $U|_{\operatorname{clos}P_{\mathcal G}\mathcal H}$ is not a unitary operator, 
then  $\sigma_{\mathrm{p}}(T^*)=\mathbb D$ and 
$\operatorname{clos}\theta(T)\mathcal H\neq\mathcal H$ for every inner function $\theta$, see Sec. 2.1.
Thus, we may assume that $U|_{\operatorname{clos}P_{\mathcal G}\mathcal H}$ is a unitary operator.  
Then  Lemma \ref{lemminimal} implies that  $\operatorname{clos}P_{\mathcal G}\mathcal H=\mathcal G$. Consequently, 
$(P_{\mathcal G}|_{\mathcal H}, U)$ is an isometric asymptote of $T$.

By Theorem IX.3.5 from \cite{sznfbk} or by \cite{k07}, there exists  $\mathcal H_1\in\operatorname{Lat}T$  
such that $T|_{\mathcal H_1}\approx S$.  By Lemma \ref{asymptote}, $P_{\mathcal G}{\mathcal H_1}$ is closed 
and $U|_{P_{\mathcal G}{\mathcal H_1}}\cong S$. Consequently,
\begin{equation*} U|_{\vee_{k=0}^\infty U^{-k}P_{\mathcal G}{\mathcal H_1}}\cong U_{\mathbb T}.
\end{equation*}
Using an appropriate unitary transformation, we may assume that 
\begin{equation*}
\vee_{k=0}^\infty U^{-k}P_{\mathcal G}{\mathcal H_1}=L^2(\mathbb T,m),\ \  U|_{L^2(\mathbb T,m)}=U_{\mathbb T},
\end{equation*}
 and
\begin{equation}\label{ppgghh1} P_{\mathcal G}{\mathcal H_1}=H^2. \end{equation}
Thus, \begin{equation}\label{gg}\mathcal G=\mathcal G_0\oplus L^2(\mathbb T,m),
\end{equation}
 where $\mathcal G_0$ is a reducing subspace for $U$, and 
\begin{equation*}\mathcal H\subset\mathcal G_0\oplus L^2(\mathbb T,m)\oplus\mathcal K.
\end{equation*}
Set $U_0=U|_{\mathcal G_0}$.

Set
\begin{equation}\label{nn1}\mathcal N=(\mathcal G_0\oplus L^2(\mathbb T,m)\oplus\mathcal K)\ominus\mathcal H 
\ \text{ and }\  
\mathcal N_1=(\mathcal G_0\oplus L^2(\mathbb T,m)\oplus\mathcal K)\ominus\mathcal H_1. 
 \end{equation}
Then
\begin{equation}\label{orthnn1} \mathcal G_0\oplus H^2_-\oplus\{0\}\subset\mathcal N_1 \ \text{ and }\  
P_{\mathcal H}(\{0\}\oplus H^2_-\oplus \{0\})\subset\mathcal N_1.
\end{equation}
Indeed, let $x\in(\mathcal G_0\oplus H^2_-\oplus\{0\})$, let $\psi\in H^2_-$, and let $y\in\mathcal H_1$. 
Then 
\begin{equation*} (x,y)=(x,P_{\mathcal G_0\oplus L^2(\mathbb T,m)}y)=(x,P_{\mathcal G}y)=0
\end{equation*}
 by \eqref{gg} and  \eqref{ppgghh1}. Furthermore, 
\begin{align*}
(P_{\mathcal H}(0\oplus \psi\oplus 0),y)&=((0\oplus \psi\oplus 0),P_{\mathcal H}y)=
((0\oplus \psi\oplus 0),y)\\&
=((0\oplus \psi\oplus 0),P_{\mathcal G}y)=0
\end{align*}
by \eqref{gg} and  \eqref{ppgghh1} again. The relation \eqref{orthnn1} is proved. 
Moreover, 
\begin{equation} \label{hh2} \mathcal N_1\cap(\{0\}\oplus H^2\oplus\{0\})=\{0\}.
\end{equation} 
Indeed, if $ h\in H^2$  is such that $0\oplus h\oplus 0\in\mathcal N_1$,  then  
\begin{equation*} (0\oplus h\oplus 0, y)=(0\oplus h\oplus 0, P_{\mathcal G}y)=0 
\end{equation*} 
for every $y\in\mathcal H_1$. The equalities \eqref{gg} and  \eqref{ppgghh1} imply 
$h\equiv 0$.

If $\{0\}\oplus H^2_-\oplus \{0\}\subset\mathcal H$, then  $\{0\}\oplus L^2(\mathbb T,m)\oplus \{0\}\subset\mathcal H$,
 because $\vee_{k=0}^\infty U_{\mathbb T}^kH^2_-=L^2(\mathbb T,m)$. Consequently, 
 there exists $\mathcal M\in \operatorname{Lat}T$ such that  $T|_{\mathcal M}\cong U_{\mathbb T}$.
Therefore, assume that  $\{0\}\oplus H^2_-\oplus \{0\}\not\subset\mathcal H$.

Take $\psi_0\in H^2_-$ such that $0\oplus\psi_0\oplus 0\not\in\mathcal H$. 
Set $y_0=P_{\mathcal H}(0\oplus\psi_0\oplus 0)$. Then there exist $x_0\in\mathcal G_0$, $\psi_1\in H^2_-$, 
$h_0\in H^2$ and  $f_0\in\mathcal K$ such that 
\begin{equation}\label{y0}y_0=x_0\oplus(\psi_1+h_0)\oplus f_0. 
\end{equation}

By \eqref{ppgghh1}, 
there exists $x\in\mathcal H_1$ such that $P_{\mathcal G}x=0\oplus 1\oplus 0\in\mathcal G_0\oplus L^2(\mathbb T,m)\oplus\mathcal K$. 
By \eqref{gg},  there exists $g_0\in\mathcal K$ such that $x=0\oplus 1\oplus g_0\in\mathcal H_1$. 
 Since  $P_{\mathcal G}|_{\mathcal H_1}$ realizes the relation $T|_{\mathcal H_1}\approx S$, 
\begin{equation}\label{spanhh1} \mathcal H_1=\vee_{k=0}^\infty
(U_0\oplus U_{\mathbb T}\oplus N)^k(0\oplus 1\oplus g_0).
\end{equation}

We may assume that $N$ has the form as in Lemma \ref{lemsubspace}. Then $f_0$ and $g_0$ have the forms 
\begin{equation*}f_0=\oplus_{n=1}^\infty f_n\in\oplus_{n=1}^\infty L^2(\Delta_n,\mu)
\ \text{ and }\  g_0=\oplus_{n=1}^\infty g_n\in\oplus_{n=1}^\infty L^2(\Delta_n,\mu),
\end{equation*}
 where $f_n$, $g_n\in L^2(\Delta_n,\mu)$, $n\geq 1$. 
Set 
\begin{equation*}
\mathrm{d}\alpha(z)=\Big(\sum_{n=1}^\infty|f_n(z)|^2+\sum_{n=1}^\infty|f_n(z)||g_n(z)|\Big)\mathrm{d}\mu(z).
\end{equation*}
Then $\alpha(\mathbb D)<\infty$. 

Let $u$ be a function from Lemma \ref{lemalpha} applied to $\alpha$, and let $\vartheta$ be a singular inner function 
from  Lemma \ref{lemest} applied to $u$. 
For every $\xi\in\mathbb T$ set 
\begin{equation}\label{uvxi} u_\xi=\vartheta_\xi(U_0) x_0\oplus\vartheta_\xi(\psi_1+h_0)\oplus  
\oplus_{n=1}^\infty\frac{f_n}{\overline\vartheta_\xi}
\ \text{ and }\   v_\xi=P_{\mathcal H}(0\oplus\vartheta_\xi \psi_0\oplus 0). 
\end{equation}
Then 
\begin{equation}\label{ker} P_{\mathcal H}u_\xi-v_\xi\in\ker\vartheta_\xi(T)^* 
\ \text{  for every }\xi\in\mathbb T.
\end{equation} 
Indeed, 
$\vartheta_\xi(T)^*=P_{\mathcal H}\vartheta_\xi(U\oplus N)^*|_{\mathcal H}$. 
Furthermore, 
\begin{align*}\vartheta_\xi(U\oplus N)^*u_\xi&=\vartheta_\xi(U_0)^*\vartheta_\xi(U_0) x_0\oplus
\overline\vartheta_\xi\vartheta_\xi(\psi_1+h_0)\oplus  \oplus_{n=1}^\infty\overline\vartheta_\xi\frac{f_n}{\overline\vartheta_\xi}\\&=
x_0\oplus(\psi_1+h_0)\oplus f_0=y_0=P_{\mathcal H}\vartheta_\xi(U\oplus N)^*u_\xi, 
\end{align*}
because $y_0\in\mathcal H$.
On the other side, 
\begin{align*}\vartheta_\xi(T)^*P_{\mathcal H}u_\xi&=P_{\mathcal H}\vartheta_\xi(U\oplus N)^*P_{\mathcal H}u_\xi
=P_{\mathcal H}\vartheta_\xi(U\oplus N)^*(u_\xi-P_{\mathcal N}u_\xi)\\&=P_{\mathcal H}\vartheta_\xi(U\oplus N)^*u_\xi,
\end{align*}
because $P_{\mathcal H}\vartheta_\xi(U\oplus N)^*P_{\mathcal N}=0$ by \eqref{nn1}. 
Furthermore, 
\begin{align*}\vartheta_\xi(T)^*v_\xi&=P_{\mathcal H}\vartheta_\xi(U\oplus N)^*v_\xi=
P_{\mathcal H}\vartheta_\xi(U\oplus N)^*P_{\mathcal H}(0\oplus\vartheta_\xi \psi_0\oplus 0)
\\&=P_{\mathcal H}\vartheta_\xi(U\oplus N)^*((0\oplus\vartheta_\xi \psi_0\oplus 0)-
P_{\mathcal N}(0\oplus\vartheta_\xi \psi_0\oplus 0))\\&=
P_{\mathcal H}(0\oplus\overline\vartheta_\xi \vartheta_\xi \psi_0\oplus 0)=
P_{\mathcal H}(0\oplus \psi_0\oplus 0)=y_0,
\end{align*}because $P_{\mathcal H}\vartheta_\xi(U\oplus N)^*P_{\mathcal N}=0$ by \eqref{nn1}. 
The relation \eqref{ker} is proved. 

Assume that $P_{\mathcal H}u_\xi-v_\xi=0$ for every $\xi\in\mathbb T$.  Then  \eqref{uvxi} implies that 
\begin{equation}\label{mainxi}
\begin{aligned}
0&=P_{\mathcal H}u_\xi-v_\xi=P_{\mathcal H}u_\xi-P_{\mathcal H}(0\oplus\vartheta_\xi \psi_0\oplus 0)\\&=
P_{\mathcal H}\Bigl(\vartheta_\xi(U_0) x_0\oplus\vartheta_\xi (\psi_1-\psi_0+h_0)\oplus  
\oplus_{n=1}^\infty\frac{f_n}{\overline\vartheta_\xi}\Bigr)\ \text{ for every }\xi\in\mathbb T.
\end{aligned}
\end{equation}

Since $\mathcal H_1\subset\mathcal H$, it follows from \eqref{spanhh1} and \eqref{mainxi} that 
  \begin{equation}\label{main1}
\begin{aligned}
&\Bigl(\vartheta_\xi(U_0) x_0\oplus\vartheta_\xi (\psi_1-\psi_0+h_0)\oplus  
\oplus_{n=1}^\infty\frac{f_n}{\overline\vartheta_\xi}, (U_0\oplus U_{\mathbb T}\oplus N)^k(0\oplus 1\oplus g_0)\Bigr)=0
\\& \text{ for every } k\geq 0 \ \text{ and every } \xi\in\mathbb T.
\end{aligned}
\end{equation}

For every $k\geq 0$ set 
\begin{equation*} \Phi_k(\xi)=\int_{\mathbb T}\vartheta_\xi (\psi_1-\psi_0+h_0)\overline\chi^k\mathrm{d}m 
\ \text{ and }\ 
G_k(\xi)=
\sum_{n=1}^\infty\int_{\mathbb D}\frac{\overline f_n g_n}{\vartheta_\xi}\chi^k\mathrm{d}\mu,  
\ \ \xi\in\mathbb D^-,
\end{equation*}
where $\chi(z)=z$ for every $z\in\mathbb D^-$. 
Then the relations \eqref{main1} can be rewritten as 
 \begin{equation}\label{main2} \Phi_k(\xi)=-\overline{ G_k(\xi)}\ \text{ for every } k\geq 0 \ \text{ and every } \xi\in\mathbb T.
\end{equation}

By Lemmas \ref{lemphitt} and  \ref{lemthetadd}, the functions $\Phi_k$ and $G_k$ are analytic in $\mathbb D$ and 
continuous in $\mathbb D^-$. Consequently,  the relations \eqref{main2} imply
\begin{equation*} \Phi_k(\xi)= \Phi_k(0)=-\overline{ G_k(0)}=-\overline{ G_k(\xi)}\ \text{ for every } k\geq 0 
\ \text{ and every } \xi\in\mathbb D^- .
\end{equation*}
The last relations and  Lemmas \ref{lemphitt} and  \ref{lemthetadd} imply 
\begin{equation*} \vartheta(0)\int_{\mathbb T} (\psi_1-\psi_0+h_0)\overline\chi^k\mathrm{d}m 
=-\frac{1}{\overline{\vartheta(0)}} \sum_{n=1}^\infty\int_{\mathbb D}
f_n\overline g_n\overline\chi^k\mathrm{d}\mu\ \text{ for every } k\geq 0 .
\end{equation*}
Taking into account \eqref{spanhh1}, we conclude  
\begin{equation}\label{orth} 0\oplus |\vartheta(0)|^2(\psi_1-\psi_0+h_0)\oplus f_0\in\mathcal N_1,
\end{equation}
where $\mathcal N_1$ is defined in \eqref{nn1}.

It follows from the definition of $y_0$ and \eqref{orthnn1} that $y_0\in\mathcal N_1$. It follows from  \eqref{orthnn1} that 
$x_0\oplus\psi_0\oplus 0\in\mathcal N_1$. These relations and \eqref{y0} imply 
\begin{equation}\label{f0} 0\oplus (\psi_1-\psi_0+h_0)\oplus f_0\in\mathcal N_1.
\end{equation}
Relations \eqref{f0} and \eqref{orth} imply 
\begin{equation*} 0\oplus(1- |\vartheta(0)|^2)(\psi_1-\psi_0+h_0)\oplus 0\in\mathcal N_1.
\end{equation*} 
The last relation and \eqref{orthnn1} imply  $0\oplus h_0\oplus 0\in\mathcal N_1$. 
By \eqref{hh2},
\begin{equation}\label{h000} h_0\equiv 0.
\end{equation}
Relations \eqref{h000}, \eqref{f0} and \eqref{orthnn1} imply 
\begin{equation*} 0\oplus0\oplus f_0\in\mathcal N_1.
\end{equation*} 
The last equality and the natural analog of Lemma \ref{lemsubspace} applied to $N^*$ imply 
\begin{equation}\label{f0xi} 0\oplus 0\oplus\oplus_{n=1}^\infty\frac{f_n}{\overline\vartheta_\xi}
\in\mathcal N_1\ \text{ for every }\xi\in\mathbb T, 
\end{equation}
because $\mathcal N_1\in\operatorname{Lat}(U\oplus N)^*$. 
It follows from  \eqref{h000} and \eqref{mainxi} that 
\begin{equation*}\vartheta_\xi(U_0) x_0\oplus\vartheta_\xi (\psi_1-\psi_0)\oplus  
\oplus_{n=1}^\infty\frac{f_n}{\overline\vartheta_\xi}\in\mathcal N_1 \ \text{ for every }\xi\in\mathbb T.
\end{equation*}
The last relation, \eqref{f0xi} and \eqref{orthnn1} imply
\begin{equation*} 0\oplus  P_{H^2}(\vartheta_\xi(\psi_1-\psi_0))\oplus 0 \in\mathcal N_1 
\ \text{ for every }\xi\in\mathbb T.
\end{equation*}
By  \eqref{hh2}, 
\begin{equation}\label{mutualmain} P_{H^2}(\vartheta_\xi(\psi_1-\psi_0))\equiv 0  
\ \text{ for every }\xi\in\mathbb T.
\end{equation}
Set $\psi_k=\overline\chi\overline h_{k0}$, $k=0,1$, where $h_{10}$, $h_{00}\in H^2$,
 and $\chi(z)=z$, $z\in\mathbb T$. Then the relation \eqref{mutualmain} 
is rewritted as 
\begin{equation*} h_{10}-h_{00}\in \vartheta_\xi H^2\ \text{ for every }\xi\in\mathbb T.
\end{equation*}
Lemma \ref{lemmutually} implies $h_{10}-h_{00}\equiv 0$.  Consequently, $\psi_1=\psi_0$. 

The last equality, the definition of $y_0$ and the equalities \eqref{y0} and \eqref{h000} imply 
\begin{equation*}y_0=x_0\oplus\psi_0\oplus f_0=P_{\mathcal H}(0\oplus\psi_0\oplus 0). 
\end{equation*}
The last equality and \eqref{nn1} imply 
\begin{equation*}x_0\oplus 0\oplus f_0=-P_{\mathcal N}(0\oplus\psi_0\oplus 0). 
\end{equation*}
Therefore, $(0\oplus\psi_0\oplus 0,P_{\mathcal N}(0\oplus\psi_0\oplus 0))=0$. This means 
$0\oplus\psi_0\oplus 0\in \mathcal H$, a contradiction with the choice of $\psi_0$. 

Thus, assuming that $P_{\mathcal H}u_\xi-v_\xi=0$ for every $\xi\in\mathbb T$, a contradiction is obtained. 
Consequently, there exists $\xi\in\mathbb T$ such that $P_{\mathcal H}u_\xi-v_\xi\neq 0$. 
By \eqref{ker}, $\ker\vartheta_\xi(T)^* \neq\{0\}$. Consequently, 
$\operatorname{clos}\vartheta_\xi(T)\mathcal H\neq\mathcal H$ 
and $\theta=\vartheta_\xi$ satisfies the conclusion of the theorem.
\end{proof}

\begin{remark} 
Let $\vartheta$ be a singular inner function from the proof of Theorem \ref{thm11}. Then the set
$\{\xi\in\mathbb T\colon \ker\vartheta_\xi(T)^* \neq\{0\}\}$ 
contains a non-empty open subarc of $\mathbb T$. Indeed, set 
 \begin{equation*}\Xi_k=\{\xi\in\mathbb T\colon\Phi_k(\xi)=-\overline{ G_k(\xi)}\}, \ k\geq 0, 
\ \text{ and }\ \Xi=\cap_{k=0}^\infty \Xi_k.
\end{equation*}
By Lemmas \ref{lemphitt} and  \ref{lemthetadd}, $\Phi_k$ and $G_k$ are continuous on $\mathbb T$. 
Consequently, $\Xi_k$ is  closed  for every $k$. Therefore, $\Xi$ is closed. 
It follows from the proof of the theorem that $\Xi\neq \mathbb T$. 
By \eqref{ker}--\eqref{main2}, 
\begin{equation*}\mathbb T\setminus\Xi\subset\{\xi\in\mathbb T\colon \ker\vartheta_\xi(T)^* \neq\{0\}\}
\end{equation*}
\end{remark}

\begin{theorem}\label{thm12} Suppose that $T$ is a non-stable (that is, not of class $C_{0\cdot}$) subnormal contraction, such that 
the spectral measure of a unitary summand of the minimal normal extension of $T$ is absolutely continuous 
(with respect of $m$). 
Then  there exists $\{0\}\neq\mathcal M_0\in \operatorname{Lat}T$ such that  $T|_{\mathcal M_0}$ is unitary  or there exists 
a singular inner function $\theta$ such that the range of $\theta(T)$ is not dense.
\end{theorem}

\begin{proof} By assumption, there exists  a Borel set $\tau\subset\mathbb T$  such that  $m(\tau)>0$ 
and a scalar-valued spectral measure of a unitary summand of the minimal normal extension of $T$ is $m|_{\tau}$. 
If  $m(\tau)=1$, then $T$ satisfies the assumption of Theorem \ref{thm11}. Thus, assume that $m(\tau)<1$.  
Set 
\begin{equation*}R=U_{\mathbb T\setminus\tau}\oplus T.
\end{equation*}
Then $R$ satisfies the assumption of Theorem \ref{thm11}. 
Consequently,  there exists $\mathcal M\in \operatorname{Lat}R$ such that  $R|_{\mathcal M}\cong U_{\mathbb T}$ or there exists 
a singular inner function $\theta$ such that the range of $\theta(R)$ is not dense.

If there exists $\mathcal M\in \operatorname{Lat}R$ such that  $R|_{\mathcal M}\cong U_{\mathbb T}$, 
then 
there exists $\{0\}\neq\mathcal M_0\in\operatorname{Lat}T$ such that  $T|_{\mathcal M_0}$ is unitary by Lemma \ref{lemtau} 
(applied to $\mathbb T\setminus\tau$; the inequality $m(\mathbb T\setminus\tau)<1$ follows from the assumption $m(\tau)>0$). 
If there exists 
a singular inner function $\theta$ such that the range of $\theta(R)$ is not dense, then the range of $\theta(T)$ is not dense, because  $\theta(R)=\theta(U_{\mathbb T\setminus\tau})\oplus \theta(T)$ and the range of 
$\theta(U_{\mathbb T\setminus\tau})$ is dense.
\end{proof}

\begin{remark} Let $T$ be a non-stable (not of class $C_{0\cdot}$) contraction, and let $(X, U)$ be a unitary asymptote of $T$. 
It is well known and easy to see that if

\noindent (a) $\mathcal H_{T,0}\neq\{0\}$, where $\mathcal H_{T,0}$ is the stable subspace of $T$ defined in \eqref{hhtt0},

\noindent or

\noindent (b) $\sigma_{\mathrm{p}}(T^*)\neq\emptyset$,

\noindent or

\noindent (c) $X^{-1}\mathcal M$ is a nontrivial subspace for some $\mathcal M\in\operatorname{Hlat}U$,

\noindent then $T$ has a nontrivial hyperinvariant subspace. 
(With respect to (c), see Theorem 30 in \cite{k16} or Theorem 1.4 in \cite{g19}; actually, (a) is a particular case of (c).)

The author does not know whether exist subnormal contractions 
which satisfy the assumptions of Theorem \ref{thm12} and do not satisfy at  least one of assumptions (a)--(c). 
\end{remark}

\section{Stable subnormal contractions}

The purpose of this section is to show that Theorem \ref{thm12} cannot be generalized to stable subnormal contractions. 
The results of this section
 are known and simple. 

\begin{lemma}\label{lem41} Let $T$ be a contraction such that  $\sigma(R)\subset\mathbb D$. Then $R$ is of class $C_{00}$, 
and $\sigma(\varphi(R))=\varphi(\sigma(R))$ for every $\varphi\in H^\infty$.
\end{lemma}

 \begin{proof} Let $r(R)<r<1$. Then $\|R^{*n}\|=\|R^n\|<r^n$ for sufficiently large $n$. Consequently, $R$ is of class $C_{00}$. 
Furthermore, the definition of $\varphi(R)$ given by $H^\infty$-functional calculus for  contractions without singular unitary summand coincides with 
the definition of $\varphi(R)$ given by the  Riesz--Dunford functional calculus, namely, 
\begin{equation*} \varphi(R)=\frac{1}{2\pi\mathrm{i}}\int_{|z|=r}\varphi(z)(z-R)^{-1}\mathrm{d}z.
\end{equation*}
The equality $\sigma(\varphi(R))=\varphi(\sigma(R))$ follows from the properties of the  Riesz--Dunford functional calculus. 
\end{proof}

\begin{proposition} \label{prop42} Let $T\in\mathcal L(\mathcal H)$ be a contraction, and let 
$\{\mathcal  M_n\}_{n=1}^\infty\subset\operatorname{Lat}T$ be such that 
\begin{equation*}\vee_{n=1}^\infty\mathcal  M_n=\mathcal H \ \text{ and } \ \sigma(T|_{\mathcal  M_n})\subset\mathbb D \ \text{ for every }n\geq 1.
\end{equation*}
Then $T$ is of class $C_{0\cdot}$. Furthermore, let $\varphi\in H^\infty$ be such that $\varphi(z)\neq 0$ for every $z\in\mathbb D$. 
Then $\operatorname{clos}\varphi(T)\mathcal H=\mathcal H$.
\end{proposition}

\begin{proof} By Lemma \ref{lem41}, $T|_{\mathcal  M_n}$ is of class $C_{00}$. 
 Consequently,  $T$ is of class $C_{0\cdot}$. Furthermore,     $\varphi(T|_{\mathcal  M_n})$ is invertible 
for every $n\geq 1$ by Lemma \ref{lem41} again. In particular, 
$\varphi(T)\mathcal  M_n =\mathcal  M_n$ for  every $n\geq 1$. Consequently, 
\begin{equation*}\operatorname{clos}\varphi(T)\mathcal H= \vee_{n=1}^\infty\varphi(T)\mathcal  M_n=
\vee_{n=1}^\infty\mathcal  M_n=\mathcal H. \qedhere
\end{equation*}
\end{proof}

\begin{remark} If $T$ satisfies the assumption of Proposition \ref{prop42}, then $T$ can be of class $C_{01}$. 
A simplest example is $T=S^*$, where $S$ is the unilateral shift, that is, the operator of multiplication by the independent variable acting on $H^2$. Indeed, let $\{\lambda_n\}_{n=1}^\infty\subset\mathbb D$ be such that 
$\lambda_n\neq\lambda_k$, if $n\neq k$,  and
$\sum_{n=1}^\infty(1-|\lambda_n|)=\infty$. Set 
\begin{equation*}h_n(z)=\frac{1}{1-\overline{\lambda_n}z}, \ z\in\mathbb D, \ \text{ and }
 \ \mathcal  M_n=\mathbb C h_n, \ n\geq 1. 
\end{equation*} Then  $\{\mathcal  M_n\}_{n=1}^\infty$ satisfy the  assumption of Proposition \ref{prop42}. 
\end{remark}

\begin{proposition}\label{prop45} Let $\mu$ be a positive finite Borel measure on $\mathbb D$.  
 Set  $\chi(z)=z$, $z\in\mathbb D$, and  
\begin{equation*} \overline{L^2_a(\mathbb D,\mu)}=\{f\in L^2(\mathbb D,\mu)\colon \overline f \ \text{ is analytic in } \mathbb D\}. 
\end{equation*}
Suppose that $\overline{L^2_a(\mathbb D,\mu)}$ is closed. 
Set  
\begin{equation*} \mathcal H=L^2(\mathbb D,\mu)\ominus\overline{L^2_a(\mathbb D,\mu)}.
\end{equation*}
Let $T$ be the operator of multiplication by $\chi$ acting on $\mathcal H$, and let $\lambda\in\mathbb D$. 
If $\overline\lambda\in\sigma_{\mathrm{p}}(T^*)$, then $\frac{1}{\chi-\lambda}\in L^2(\mathbb D,\mu)$.  
\end{proposition}

\begin{proof} 
Let $g$ be a function analytic in $\mathbb D$. Set $g_\lambda(z)=\frac{g(z)-g(\lambda)}{z-\lambda}$, if $z\neq\lambda$, 
$z\in\mathbb D$, and $g_\lambda(\lambda)=g'(\lambda)$. If $g\in L^2(\mathbb D,\mu)$, then $g_\lambda\in L^2(\mathbb D,\mu)$.  

Let $f\in \mathcal H$, let $f\not\equiv 0$, and let $T^*f=\overline\lambda f$. 
Then $(\overline\chi-\overline\lambda)f\in\overline{L^2_a(\mathbb D,\mu)}$. Consequently, there exists $g\in L^2(\mathbb D,\mu)$ 
such that $g$ is analytic in $\mathbb D$ and $(\overline\chi-\overline\lambda)f=\overline g$. 
Thus, $f=\overline{g_\lambda}+\frac{\overline{g(\lambda)}}{\overline\chi-\overline\lambda}$. 
 If $g(\lambda)=0$, then $f\in\overline{ L^2_a(\mathbb D,\mu)}$. 
This implies  $f\equiv 0$, a contradiction. Thus,   $g(\lambda)\neq 0$. 
Consequently,  $\frac{1}{\chi-\lambda}\in L^2(\mathbb D,\mu)$. 
\end{proof}

Recall that $m$ is the normalized Lebesgue measure on the unit circle $\mathbb T$. 

\begin{theorem} \label{thm44} Let $\nu$ be a positive finite Borel measure on $(0,1)$ such that  $1\in\operatorname{supp}\nu$. 
 Set 
$\mathrm{d}\mu(r\zeta)=\mathrm{d}m(\zeta)\mathrm{d}\nu(r)$, where $r\in(0,1)$ and $\zeta\in\mathbb T$, 
 and set $\chi(z)=z$, $z\in\mathbb D$. 
Set 
\begin{equation*} \overline{P^2(\mathbb D,\mu)}=\vee_{n=0}^\infty\overline{\chi^n}
 \ \text{ and } \ \mathcal H=L^2(\mathbb D,\mu)\ominus\overline{P^2(\mathbb D,\mu)}.
\end{equation*}
Then  
 \begin{equation*} \overline{P^2(\mathbb D,\mu)}=\overline{L^2_a(\mathbb D,\mu)}, 
\end{equation*}
where $\overline{L^2_a(\mathbb D,\mu)}$ is defined in Proposition \ref{prop45}.
Let $T$ be the operator of multiplication by $\chi$ acting on $\mathcal H$.
Then $T$ is a pure subnormal contraction, $T$ satisfies the assumption of Proposition \ref{prop42}, $\mathbb T\subset\sigma(T)$,  and $T\cong \xi T$ for every $\xi\in\mathbb T$. If, in addition, 
\begin{equation}\label{infty}\int_{(0,1)}\frac{\mathrm{d}\nu(r)}{|r^2-a^2|}=\infty \ \text{ for every } a\in[0,1),  
\end{equation}
then $\sigma_{\mathrm{p}}(T^*)=\emptyset$.
\end{theorem}

\begin{proof}  Since $1\in\operatorname{supp}\nu$, we have  $\mathbb T\subset\operatorname{supp}\mu$   
and $\lim\|\chi^n\|^{\frac{1}{n}}=1$. Furthermore, 
\begin{equation}\label{equal}\int_{\mathbb D}\frac{\mathrm{d}\mu}{|\chi-\lambda|^2} = 
\int_{(0,1)}\frac{\mathrm{d}\nu(r)}{|r^2-|\lambda|^2|} 
\end{equation}
for every $\lambda\in\mathbb D$.

Denote by $N_\mu$ the operator of multiplication by $\chi$ acting on $L^2(\mathbb D,\mu)$. 
Then  $\overline{P^2(\mathbb D,\mu)}\in\operatorname{Lat}N_\mu^*$. Consequently, 
$\mathcal H\in\operatorname{Lat}N_\mu$. Thus, $T$ is subnormal. Furthermore, $T$ is pure by Proposition 3.4 in \cite{feldman}. 

It is easy to see that $N_\mu$ is the minimal normal extension of $T$. By Theorem II.2.11(a) in \cite{conwaysubnormal}, 
 $\sigma(N_\mu)\subset\sigma(T)$.
Consequently,  $\mathbb T\subset\sigma(T)$. 
  
 Since $\lim\|\chi^n\|^{\frac{1}{n}}=1$, and    $\{\chi_n\}_{n=1}^\infty$ is an orthogonal family,  we have 
 \begin{align*} P^2(\mathbb D,\mu)&=\vee_{n=0}^\infty\chi^n=
\{f \colon f \text{ is analytic in } \mathbb D \text{ and } 
\sum_{n=0}^\infty|\widehat f(n)|^2\|\chi^n\|^2<\infty\}\\&=
\{f \colon f \text{ is analytic in } \mathbb D \text{ and }   f\in L^2(\mathbb D,\mu)\}. 
\end{align*}
Thus, $\overline{P^2(\mathbb D,\mu)}=\overline{L^2_a(\mathbb D,\mu)}$. 
If \eqref{infty} is fulfilled, then  
 $\sigma_{\mathrm{p}}(T^*)=\emptyset$ by \eqref{equal} and  Proposition \ref{prop45}.  

Let $\xi\in\mathbb T$. Define $W_\xi\in\mathcal L(L^2(\mathbb D,\mu))$ by the formula $(W_\xi f)(z)=f(z\overline\xi)$, 
$z\in\mathbb D$.  
Then $W_\xi$ is a unitary trasformation, $W_\xi \xi N_\mu= N_\mu W_\xi$, and 
$W_\xi\overline{\chi^n}=\xi^n\overline{\chi^n}$ for every $n\geq 0$. Therefore, 
\begin{equation*} W_\xi \overline{P^2(\mathbb D,\mu)}= \overline{P^2(\mathbb D,\mu)}.
\end{equation*}
Consequently, $W_\xi\mathcal H=\mathcal H$. Thus, $W_\xi$ realizes the relation $T\cong\xi T$. 

Set \begin{equation*} 
 e_n=\frac{\overline{\chi^n}}{\|\chi^n\|}, \ \ n\geq 0. 
\end{equation*}
Then   $\{e_n\}_{n=1}^\infty$ is an orthonormal basis of $ \overline{P^2(\mathbb D,\mu)}$. 

We have
 \begin{equation*} L^2(\mathbb D,\mu)= 
\bigvee_{n\geq 0\atop k\geq 1}|\chi|^{2n}\chi^k\vee \bigvee_{n\geq 1\atop k\geq 0}|\chi|^{2n}\overline{\chi^k}
\vee \overline{P^2(\mathbb D,\mu)}.
\end{equation*}
It is easy to see that
\begin{equation*} \bigvee_{n\geq 0 \atop k\geq 1}|\chi|^{2n}\chi^k\subset\mathcal H.  
\end{equation*}
Furthermore, for every $n\geq 1$, $k\geq 0$ we have 
\begin{align*} P_{\mathcal H}|\chi|^{2n} \overline{\chi^k}&=|\chi|^{2n} \overline{\chi^k}
-P_{\overline{P^2(\mathbb D,\mu)}}|\chi|^{2n}\overline{\chi^k}=
 |\chi|^{2n} \overline{\chi^k}-(|\chi|^{2n} \overline{\chi^k}, e_k)e_k\\&
=\Bigl(|\chi|^{2n}-\frac{\|\chi^{n+k}\|^2}{\|\chi^k\|^2}\Bigr)\overline{\chi^k}. 
\end{align*}
Thus, 
\begin{equation}\label{4hh} \mathcal H=\bigvee_{n\geq 0\atop k\geq 1}|\chi|^{2n}\chi^k\vee 
\bigvee_{n\geq 1\atop k\geq 0}\Bigl(|\chi|^{2n}-\frac{\|\chi^{n+k}\|^2}{\|\chi^k\|^2}\Bigr)\overline{\chi^k}.
\end{equation}

For every $0<r<1$ set $D_r=\{|z|< r\}$ and $D_r^-=\{|z|\leq r\}$. As usual, set  
  \begin{equation*} L^2(D_r,\mu)=\{f\in L^2(\mathbb D,\mu)\colon f(z)=0 \ \text{ for } z\in \mathbb D\setminus D_r\}. 
\end{equation*}
Let   $\chi_r(z)=z$ for $z\in D_r$ and $\chi_r(z)=0$ for $z\in \mathbb D\setminus D_r$. 
 Set 
 \begin{equation*}\overline{P^2(D_r,\mu)}=\vee_{n=0}^\infty\overline{\chi_r^n}
 \ \text{ and } \ \mathcal H_r=L^2(D_r,\mu)\ominus\overline{P^2(D_r,\mu)}.
\end{equation*}
Since $(f,\overline{\chi^n})=(f,\overline{\chi_r^n})$ for every $f\in L^2(D_r,\mu)$, 
we conclude $\mathcal H_r\subset \mathcal H$. It is easy to see that 
  $\mathcal H_r\in\operatorname{Lat}T$ and $\sigma(T|_{\mathcal H_r})\subset D_r^-$. 

As for $\mathcal H$, we have  
\begin{equation}\label{4hhr} \mathcal H_r=\bigvee_{n\geq 0\atop k\geq 1}|\chi_r|^{2n}\chi_r^k\vee 
\bigvee_{n\geq 1\atop k\geq 0}\Bigl(|\chi_r|^{2n}-\frac{\|\chi_r^{n+k}\|^2}{\|\chi_r^k\|^2}\Big)\overline{\chi_r^k}.
\end{equation}

Denote by $f_r$ some function from the right side of \eqref{4hhr} and by $f$ the corresponding function 
 from the right side of \eqref{4hh}. It follows from the Lebesgue convergence theorem that $f_r\to f$ when $r\to 1$. 
Let $\{r_n\}_{n=1}^\infty\subset (0,1)$  be such that $r_n\to 1$. Set $\mathcal M_n=\mathcal H_{r_n}$. 
The equalities \eqref{4hhr} and  \eqref{4hh} imply $\vee_{n=1}^\infty\mathcal  M_n=\mathcal H $. 
Consequently, $T$ satisfies the assumption of Proposition \ref{prop42}.
\end{proof}

\begin{corollary} \label{cor44} Let $T\in\mathcal L(\mathcal H)$ be as in Theorem \ref{thm44}, and let \eqref{infty}
be fulfilled. Then $\operatorname{clos}\varphi(T)\mathcal H=\mathcal H$ for every
 $\varphi\in H^\infty$ such that $\varphi\not\equiv 0$. 
\end{corollary}

\begin{proof} Let $\varphi\in H^\infty$ be  such that $\varphi\not\equiv 0$. If $\varphi(z)\neq 0$ for every
 $z\in\mathbb D$, then $\operatorname{clos}\varphi(T)\mathcal H=\mathcal H$ by  Proposition \ref{prop42}. 
If there exists  $z\in\mathbb D$ such that  $\varphi(z)=0$, then  $\operatorname{clos}\varphi(T)\mathcal H=\mathcal H$ by Lemma 1.2 in \cite{g19}, because $\sigma_{\mathrm{p}}(T^*)=\emptyset$. 
\end{proof}

\begin{remark} Recall that a subnormal operator $T$ has Dunford's property, that is, the local spectral subspaces 
 for $T$ corresponding to closed subsets of $\mathbb C$ are closed. The definition of  local spectral subspace 
is not recalled here, see \cite{laursenneumann}. 
 Let $T$ be as in Theorem \ref{thm44}. Then there exists $r\in(0,1)$ such that 
$\{0\}\neq\mathcal H_r\neq\mathcal H$, where $\mathcal H_r$ is from \eqref{4hhr}. 
Since $\sigma(T)\setminus\sigma(T|_{\mathcal H_r})\neq\emptyset$, we have $T$ has a nontrivial 
hyperinvariant subspace.  Namely, this is the local spectral subspace for $T$ corresponding to 
$\sigma(T|_{\mathcal H_r})$. For details of the proof see, for example, Proposition 3.6 in \cite{jungko}. 
\end{remark}

\begin{remark} Let in assumption of Theorem \ref{thm44} $\mathrm{d}\nu(r)=2r\mathrm{d}r$. 
Then $\mu$ is the normalized area measure on $\mathbb D$ and $T$ is decomposable by \cite{afv}.  
The definition of decomposable operator is not recalled here, see \cite{laursenneumann}. 
For other examples of decomposable (non-normal) subnormal operators see \cite{albr} and \cite{radj}. 
\end{remark}

\end{document}